\documentclass[a4paper,UKenglish,cleveref, autoref,
thm-restate,authorcolumns]{lipics-v2019} 

\title{Abstract rewriting internalized} 

\author{Maxime LUCAS}{Inria and LS2N, Universit\'e de Nantes, France
\and \url{http://lucas-webpage.gforge.inria.fr/}
}{Maxime.Lucas@inria.fr}{https://orcid.org/0000-0002-9680-7129}{}

\authorrunning{M. Lucas} 

\Copyright{Maxime Lucas} 



\category{} 

\relatedversion{} 

\acknowledgements{I thank Mathieu Anel, Cameron Calk, and Damien Pous for
  fruitful conversations during the preparation of this article, Marie Kerjean
  for helpful advice to make this article more accessible, and Cyrille Chenavier
  for both.}

\nolinenumbers 

\hideLIPIcs 

\EventEditors{John Q. Open and Joan R. Access} \EventNoEds{2} \EventLongTitle{42nd Conference on Very Important
Topics (CVIT 2016)} \EventShortTitle{CVIT 2016} \EventAcronym{CVIT}
\EventYear{2016} \EventDate{December 24--27, 2016} \EventLocation{Little
Whinging, United Kingdom} \EventLogo{} \SeriesVolume{42} \ArticleNo{23}

\usepackage{mathtools}
\usepackage{accents}

\usepackage{xargs}
\usepackage{tikz}
\usetikzlibrary{matrix, cd, positioning}

\newlength{\myline}
\newcommandx*{\doublearrow}[2]{
  \draw[line width=rule_thickness,double equal sign distance,#1] #2;
}
\newcommandx*{\triplearrow}[4][1=0, 2=1]{
  \draw[line width=\myline,double distance=5\myline,#3] #4;
  \draw[line width=\myline,shorten <=#1\myline,shorten >=#2\myline,#3] #4;
}
\newcommandx*{\quadarrow}[4][1=0, 2=2.5]{
  \draw[line width=\myline,double distance=8\myline,#3] #4;
  \draw[line width=\myline,double distance=2\myline,shorten <=#1\myline,shorten >=#2\myline,#3] #4;
}

\makeatletter \let\orgdescriptionlabel\descriptionlabel
\renewcommand*{\descriptionlabel}[1]{%
  \let\orglabel\label \let\label\@gobble
  \phantomsection \protected@edef\@currentlabel{#1}%
  \let\label\orglabel
  \orgdescriptionlabel{#1}%
} \makeatother

\newcommand\refitem[1]{\textcolor{lipicsGray}{\sffamily\bfseries\mathversion{bold}\ref{#1}}}

\newcommand\source[1]{\sigma_{#1}}
\newcommand\target[1]{\tau_{#1}}
\newcommand\minimum[1]{#1_{min}}

\newcommand\wfto[1]{\xrightarrow[wf]{#1}}
\newcommand\algto[1]{\xrightarrow[alg]{#1}}
\newcommand\Filt[1]{#1\text{-}\operatorname{Filt}}
\newcommand\sym{\mathcal S}

\DeclareMathOperator{\Graph}{Gph}

\DeclareMathOperator{\Ab}{Ab}
\DeclareMathOperator{\colim}{colim}
\DeclareMathOperator{\Set}{Set}

\DeclareMathOperator{\Vect}{Vect}
\DeclareMathOperator{\Gp}{Gp}
\DeclareMathOperator{\supp}{supp}
\DeclareMathOperator{\height}{ht}

\tikzset{
    labl/.style={anchor=south, rotate=45, inner sep=.5mm}
}

\begin{document}

\maketitle

\begin{abstract}
  In traditional rewriting theory, one studies a set of terms up to a set of
  rewriting relations. In algebraic rewriting, one instead studies a vector
  space of terms, up to a vector space of relations. Strikingly, although both
  theories are very similar, most results (such as Newman's Lemma) require
  different proofs in these two settings.

  In this paper, we develop rewriting theory internally to a category
  $\mathcal C$ satisfying some mild properties. In this general setting, we
  define the notions of termination, local confluence and confluence using the
  notion of reduction strategy, and prove an analogue of Newman's Lemma. In the
  case of $\mathcal C= \Set$ or $\mathcal C = \Vect$ we recover classical
  results of abstract and algebraic rewriting in a slightly more general form,
  closer to von Oostrom's notion of decreasing diagrams. 
\end{abstract}


\section{Introduction}

The goal of this work is to bridge the gap between two major branches of
rewriting theory: abstract rewriting, stemming from the work of Newman, and
algebraic rewriting. In algebraic rewriting, convergent presentations are called
Gr\"obner basis and were introduced by Buchberger to compute basis of algebras. In
particular, it allows one to solve the ideal membership problem, which is the
linear equivalent of the word problem. Today Gr\"obner basis are in particular
used in control theory.

The two theories are very similar: both define a notion of terminating,
confluent or locally confluent relation, and prove a diamond Lemma showing that
locally confluent terminating system are confluent (in the case of algebraic
rewriting, this was explicitly done by Bergman~\cite{Bergman78}). Nevertheless,
we are not aware of any treatment unifying those two theories. 

The main difference between algebraic and abstract rewriting is that algebraic
rewriting is not stable by contextual closure. For example, suppose given an
rewriting relation $\to$ on a vector space. If $u \to v$ is a valid rewriting step,
then $u + w \to v + w$ is not necessarily valid. This is necessary to avoid
non-terminating behaviours: otherwise the rewriting step $u - u - v \to v -u -v$
would be valid, but this is just $-v \to -u$, which implies that $v \to u$: in other
words, $\to$ is always symmetric and can never be terminating. The failure to take
this phenomenon into account plagued many early papers studying the algebraic
$\lambda$-calculus, such as \cite{ER_diff_lambda}, and was not recognized before
\cite{V_algebraic_lambda}.

While the existence of two rewriting theories is not an issue \textit{per se},
the multiplication of applications of higher dimensional rewriting to various
algebraic structures calls for a unified framework. Higher dimensional rewriting
seeks to apply rewriting techniques to study homotopical and homological
properties of algebraic objects. Already existing examples include monoids
\cite{SOK94}, algebras \cite{GHM_algebra}, string diagrams (encoded as Pros and
ProPs) \cite{GM_Pro}, and term rewriting systems \cite{MM_term}. Future cases of
interest include e.g. (non-symmetric, symmetric or shuffle, linear or
set-theoretic) operads or linear Pros.

All these structures can be represented as monoids inside a category
$\mathcal C$: taking $\mathcal C = \Set$ we obtain (regular) monoids, monoid
objects in $\mathcal C = \Vect_k$ are $k$-algebras, and the different flavors of
operads are all monoids in various categories of collections. This paper
constitutes a first step towards a unified treatment of higher dimensional
rewriting for these different objects, by developing a general theory of
rewriting inside a category $\mathcal C$, omitting with the monoid structure for
now. In the $\mathcal C = \Set$ and $\mathcal C = \Vect_k$ we recover
respectively abstract and algebraic rewriting, in a slightly more general form
than the one usually presented. Those two cases have already presented in
earlier works, namely \cite{These} for the case $\mathcal C = \Set$, and
\cite{CL_IWC} for the case $\mathcal C = \Vect_k$.

\medskip

There are two ways to model a relation in abstract rewriting. The first one is
to see a relation $\to$ on an object $E$ as a subset of $E \times E$. The second one,
which is the one suited for higher dimensional rewriting and that we will
generalize here, is to see a relation as a set $R$ equipped with maps
$\source R, \target R : R \to E$, associating to any $f \in R$ its source
$\source R f$ and target $\target R f$, and making $(E,R,\source R, \target R)$
a directed graph\footnote{This can also be seen as a \emph{labelled} rewriting
  relation.}. We therefore study the rewriting properties of graphs internal to
an arbitrary category $\mathcal C$ satisfying some mild properties. We now
describe some of the content of the paper.

While the notions of reflexivity, symmetry and transitivity can be defined
straightforwardly in this context, the main obstacle lies in defining an
appropriate notion of termination for a graph $R$ on an object $E$. While in
abstract rewriting a relation $R$ is \textit{intrinsically} terminating, the
situation is more subtle in algebraic rewriting. There, a relation on a vector
space $E = kX$ is said to terminate \emph{with respect to a terminating order on
  $X$}. In practice, $kX$ is often the vector space underlying a polynomial
algebra, and the terminating order is a monomial order.

In an arbitrary category $\mathcal C$, we proceed similarly, by supposing that
$E$ is endowed with a filtration
$E_0 \subset E_1 \subset E_2 \subset \ldots \subset E$, which encodes the terminating ordering on
$E$. Formally, such a filtration is given by a terminating order $(I,\leq)$ and a
functor $I \to \mathcal C$ whose colimit is $E$. From this we can define the
\emph{object of normal forms} of $E$, denoted $\minimum E$, as the union of the $E_i$
where $i$ is a minimal element of $I$.

A graph $R$ on $E$ is then terminating if it is compatible with this filtration
in a suitable way, expressed through the existence of a local strategy
$h : E \to R$. When $E$ is a set, a local strategy maps any $x \in E$ to a rewriting
step $h(x)$ of source $x$ and whose target is smaller than $x$. Any local
strategy induces a (global) strategy $H$, mapping any $x$ to a path $H(x)$ to
(one of) its normal form, denoted $H^\tau(x)$. The end diagram is the following:

 \[
   \begin{tikzcd}[column sep = large]
     R \ar[r, "\source R", yshift=.1cm] \ar[r, ,yshift=-.1cm, "\target R"']
     &
     E \ar[r, "H^\tau"]
     \ar[l,bend right = 50, "H"']
     \ar[r, bend left = 50, hookleftarrow, "\iota"]
     &
     \minimum E
   \end{tikzcd}
 \]
 and it satisfies the relations:
 \[
   \source R \circ H = id_E \qquad
   \target R \circ H = \iota \circ H^{\tau} \qquad
   H^{\tau} \circ \iota = Id_{\minimum E}
 \]

 This almost makes the previous diagram into a split coequaliser, the only
 equation missing being $H^\tau \circ \source R = H^\tau \circ \target R$. We define this as
 our notion of confluence, which implies immediately that if $R$ is a
 terminating and confluent relation, then the quotient $E/R$ is isomorphic to
 the object of normal forms $\minimum E$. When $\mathcal C$ is a locally
 finitely presented category, We also define a notion of local confluence and
 show that, together with termination, it implies confluence, recovering
 Newman's Lemma in this general setting.
 
\paragraph*{Organisation}

In Section \ref{sec:abs_alg_rew}, we start by recalling some classical
definitions and results of abstract and algebraic rewriting. Those will be
useful throughout the paper in order to compare them to the methods developed in
the subsequent sections.

In Section \ref{sec:graphs_as_relations}, we investigate elementary properties
of graphs internal to a category $\mathcal C$, such as reflexivity, symmetry and
transitivity. We also define the notions of reflexive, symmetric and transitive
closures of a relation.

Section \ref{sec:termination} is devoted to the notion of termination of a
graph.  We define termination of a relation as the existence of a local strategy
compatible with this filtration. We finish this section by examples, showing in
particular that for any terminating (abstract or algebraic) relation induces a
terminating graph.

In Section \ref{sec:confluence}, we define a notion of global strategy for a
graph $R$ on a filtered object $E$, and define when such a strategy is
confluent. We show that whenever there exists a confluent strategy, then the
quotient of $E$ by $R$ is isomorphic to the object of normal forms. We show that
any local strategy induces a global strategy, although not necessarily a
confluent one.

Finally in Section \ref{sec:loc_to_glob_confluence} we restrict ourselves to the
case when $\mathcal C$ is a locally finitely presented category. In this case,
we are able to give a criterion for a local strategy to induce a confluent
strategy. Interpreting this criterion as a form of local confluence, we obtain a
general proof of Newman's and Bergman's diamond Lemmas.


\section{Abstract and algebraic rewriting}\label{sec:abs_alg_rew}

In this section, we recall some of the standard results of abstract and
algebraic rewriting. This will be useful in subsequent sections for comparison
with our general theory. Since in Section \ref{sec:confluence} confluence will
only be defined on terminating graphs, we do not dwell on the case of confluent
but non-terminating relations. None of the results presented here are new,
although perhaps algebraic rewriting is only rarely presented at this level of
generality.

Contrary to the rest of this article, where we will model relations by graphs,
here we stick to the more usual presentation of subsets of $E \times E$. While this
choice is somewhat arbitrary, it will be useful later on in order to distinguish
between termination in the sense of relations, and termination in the sense of
graphs.

\begin{definition}\label{def:set_relations}
  Let $E$ be a set. By a (set-theoretic) relation, we mean a subset of
  $E \times E$. We say that such a relation $\to$ is \emph{terminating} if there exists
  no infinite sequence $a_0,a_1,a_2,\ldots \in E$ such that $a_0 \to a_1 \to a_2 \to \ldots $.

  An element $e \in E$ is said to a normal form for $\to$ if there exists no
  $y \in E$ such that $x \to y$. We denote by $NF(\to)$ the set of normal forms for $\to$.

  We denote by $\xrightarrow =$, $\xrightarrow +$, $\xrightarrow *$ and
  $\xleftrightarrow *$ respectively the reflexive, transitive,
  reflexive-transitive and symmetric-reflexive-transitive closure of $\to$.  

  We say that $\to$ is confluent if for any $u,v,w$ such that
  $u \xrightarrow * v$ and $u \xrightarrow * w$, there exists $z$ such that
  $v \xrightarrow * z$ and $w \xrightarrow * z$.

  We say that $\to$ is locally confluent if for any $u,v,w$ such that
  $u \rightarrow v$ and $u \rightarrow w$, there exists $z$ such that
  $v \xrightarrow * z$ and $w \xrightarrow * z$.

  Finally, we denote by $E/\xleftrightarrow *$ the quotient of $E$ by the equivalence relation
  generated by~$\to$.
\end{definition}

\begin{proposition}\label{prop:set_equiv_confluence}
  Let $\to$ be a terminating relation on a set $E$. The following are equivalent:
  \begin{description}
  \item[SC1\label{itm:set_local_conf}] The relation $\to$ is locally confluent.
  \item[SC2\label{itm:set_conf}] The relation $\to$ is confluent.
  \item[SC3\label{itm:set_CR_prop}] For any $u,v \in E$ such that
    $u \xleftrightarrow * v$, there exists $w$ such that $u \xrightarrow * w$
    and $v \xrightarrow * w$.
  \item[SC4\label{itm:set_nf}] The canonical map
    $NF(\to) \longrightarrow (E/\xleftrightarrow *)$, sending any normal form to its equivalence
    class modulo $\xleftrightarrow *$, is a bijection.
  \end{description}
\end{proposition}



\begin{remark}
  Property \refitem{itm:set_CR_prop} above is known as the Church-Rosser
  property, and is equivalent to confluence even without the hypothesis that
  $\to$ is terminating.
\end{remark}

Throughout this article, we fix $k$ a characteristic $0$ field. We now give a
quick presentation of algebraic rewriting.

\begin{definition}
  Let $X$ be a set and let $kX$ denote the vector space spanned by $X$. Any
  $u \in kX$ can be written in a unique way as $\sum_{i=1}^n \lambda_i x_i$, where
  $\lambda_i \neq 0$ and $x_i \in X$. The set $\{x_1,\ldots,x_n\}$ is called the \emph{support} of
  $u$ and is denoted by $\supp(u)$.

  An \emph{algebraic relation} $\to$ on $kX$ is a subset of $X \times kX$ such that
  whenever $x \to u$, then $x \notin \supp (u)$.  Such a relation induces two
  (set-theoretic) relations on $kX$, denoted $\algto{}$ and
  $\wfto{}$. The first is defined by
  $\lambda x + v \algto{} \lambda u + v$ for all $\lambda \in k$,
  $u,v \in kX$ and $x \in X$ such that $x \to u$.

  The second one is the restriction of $\algto{}$ to the case when
  $\lambda \neq 0$ and $x \notin \supp(v)$.

  Let $\leq$ be a terminating order on $X$. We say that $\to$ is \emph{terminating
    with respect to $\leq$} if whenever $x \to u$, then any $y$ appearing in the
  support of $u$ is strictly smaller than $x$. This implies in particular that
  $\wfto{}$ is terminating, as a set-theoretic relation.

  The set of normal forms for $\wfto{}$ forms a sub-vector space of
  $kX$, a basis of which is given by elements $x \in X$ such that there exists no
  $u \in kX$ such that $x \to u$. The vector space of normal forms is denoted
  $NF(\to)$.

  Finally, we denote by $\xleftrightarrow *$ the congruence generated by $\to$
  (which is also the equivalence relation generated by $\algto{}$).
\end{definition}

\begin{note}
  The use of a terminating order on $X$ to define the termination of an
  algebraic relation above is fairly standard in the literature (and is the one
  used by Bergman in \cite{Bergman78}), but is actually not really necessary:
  one could instead request that $\wfto{}$ is terminating, as a set-theoretic
  relation. This is done for example in \cite[Section3.2]{GHM_algebra}.

  Although this allows one to deduce many results on algebraic rewriting from
  abstract rewriting, we argue that it is still not satisfying: the definition
  of the relation $\wfto{}$ is very specific to the category at hand (here
  $\mathcal C = \Vect_k$) and cannot be generalised to other categories, such as
  the category of groups.
\end{note}

The following lemma clarifies the relationship between the relations $\to$,
$\wfto{}$ and $\algto{}$.
\begin{lemma}
  Let $X$ be a set and $\to$ be an algebraic relation on $kX$.
  \begin{itemize}
  \item For any $x \in X$, the relations $x \to u$ and $x \wfto{} u$ are
    equivalent.

  \item If $u \algto{} v$, there exists $w \in kX$ such that
    $ u \wfto{=} w \xleftarrow[wf]{=} v $
    \end{itemize}
\end{lemma}
\begin{proof}
  While the first point is an easy verification, the second is more subtle. A
  proof can be found in the proof of \cite[Theorem 4.2.1]{GHM_algebra}.
\end{proof}

\begin{proposition}
  Let $X$ be a set equipped with a terminating order $\leq$, and $\to$ be an
  algebraic relation on $kX$ which is terminating with respect to $\leq$. The
  following are equivalent:
  \begin{description}
  \item[AC1\label{itm:A_local_conf}] For any $x \in X$, if $x \to u$ and
    $x \to v$ then there exists $w$ such that $u \wfto * w$ and
    $v \wfto{*} w$.
  \item[AC2\label{itm:A_conf}] The relation $\wfto{}$ satisfies the
    equivalent properties of Proposition \ref{prop:set_equiv_confluence}.
  \item[AC3\label{itm:A_nf}] The canonical map
    $NF(\to) \longrightarrow (kX/\xleftrightarrow *)$ sending any normal form to its equivalence
    class modulo $\xleftrightarrow *$, is an isomorphism.
  \end{description}
\end{proposition}


\begin{example}
  Take $kX = k[x]$: in other words, $X = \{1,\ldots,x^n,\ldots\}$. We can equip $X$ with the
  terminating order induced by $x^i < x^{i+1}$. Given a unitary polynomial
  $P \in k[x]$ we can write $P = x^n + \sum_{i = 0}^{n-1} x^i$. This induces a
  (confluent and terminating) algebraic rewriting system
  $x^n \to - \sum_{i = 0}^{n-1} x^i$.

  The quotient $kX/\xleftrightarrow *$ is none other than $k[X]/(P)$, while
  $NF(\to)$ is the subspace of polynomials of degree at most $(n-1)$. Taking
  $k = \mathbb R$ and $P = x^2 + 1$, the isomorphism of property
  \refitem{itm:A_nf} is none other than the canonical isomorphism between
  $\mathbb R + x \mathbb R$ and $R[x]/(x^2+1) = \mathbb C$.
\end{example}

\section{Graphs as relations}\label{sec:graphs_as_relations}

We fix a category $\mathcal C$, and suppose that $\mathcal C$
is finitely complete has all countable (including finitary) coproducts.

In this section, we study define some elementary properties of graphs, seen as
generalized relations.



\begin{definition}
  Let $\mathcal C$ be a category, and $E$ an object of $\mathcal C$. We
  denote by $\Graph_E$ the category of graphs over $E$.
  Objects are triples $(R,\source R,\target R)$ , where $R$ is an object of
  $\mathcal C$ and $\source R : R \to E$ and $\target R : R \to E$ are maps in $\mathcal C$.

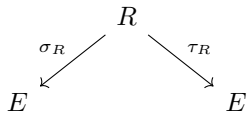
  \[
    \begin{tikzcd}
      & R \arrow[ld, "\source R"'] \arrow[rd, "\target R"] & \\
      E & & E
    \end{tikzcd}
  \]

  A morphism of graphs from $(R,\source R,\target R)$ to $(S,\source S,\target S)$
  is a map $R \to S$ commuting with $\source{}$ and $\target{}$.

  We will often denote a graph $(R,\source R,\target R)$ simply by $R$.
\end{definition}

\begin{definition}\label{defn:operations_on_graphs}
  Let $E$ be an object of $\mathcal C$, and let $R,S$ be two graphs over $E$. 
  \begin{itemize}
  \item $E$ canonically inherits the structure of a graph using
    $\source E = id_E$ and $\target E = id_E$.
  \item The product of $R$ and $S$, denoted $RS$, is the graph defined by
    the following pullback, with $\source{RS} = \source R \circ \pi_1$ and
    $\target{RS} = \target R \circ \pi_2$:
    
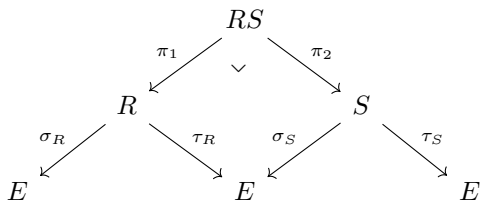
\[
      \begin{tikzcd}
        & & RS \ar[ld, "\pi_1"'] \ar[rd, "\pi_2"] \ar[dd, phantom, "\llcorner" labl , near start] & & \\
        & R \ar[ld, "\source R"'] \ar[rd, "\target R"]& & S \ar[ld, "\source S"'] \ar[rd, "\target S"] & \\
        E & & E & & E
      \end{tikzcd}
    \]
  \item The sum of $R$ and $S$, denoted $R + S$, is the graph
    $(R \coprod S,\source R \coprod \source S, \source R \coprod \source S)$, where
    $\coprod$ denotes the coproduct in $\mathcal C$.
  \item The transitive closure of $R$, denoted $R^+$, is the graph over $E$ with
    underlying object the countable coproduct
    $\coprod_{i = 1}^{\infty} R^{\times_E i}$, with $\source {R^+}$ (resp.
    $\target {R^+}$) defined on the component $R^{\times_E i}$ as the composite
    $\source R \circ \pi_1$ (resp.  $\target R \circ \pi_i$) for any $i \geq 1$.
  \item The transitive reflexive closure of $R$, denoted $R^*$, is the graph
    $R^+ + E$.
  \item The opposite of $R$, denoted $R^{\circ}$, is the graph obtained by
    reversing the source and target of $R$: $\source {R^\circ} = \target R$ and
    $\target {R^\circ} = \source R$.
  \item The reflexive symmetric transitive closure of $R$, denoted $R^\sym$, is
    the graph $(R+R^\circ)^*$.
  \end{itemize}
\end{definition}

\begin{definition}
  Let $R$ be a graph over an object $E$ in a category $\mathcal C$.
  We say that $R$ is:
  \begin{itemize}
  \item reflexive if there exists a map (of graphs) $u_R : E \to R$,
  \item transitive if there exists a map $m_R : R R \to R$,
  \item symmetric if there exists a map $s_R: R^\circ \to R$.
  \end{itemize}
\end{definition}

The closure terminology of Definition \ref{defn:operations_on_graphs} is
justified by the following lemma:
\begin{lemma}\label{lem:sanity_check}
  Let $R$ be a graph over an object $E$ in a category $\mathcal C$.  Then
  $R + E$ is a universal among the reflexive graphs under $R$, in the sense that
  $R + E$ is reflexive, and for any map $f : R \to S$ in $\Graph_E$, where $S$ is
  a reflexive graph, then there exists a unique map $\tilde f : R + E \to S$
  preserving the reflexive structure such that the composite with the inclusion
  $R \to R+E$ yields $f$.

  Similarly:
  \begin{itemize}
  \item The graph $R + R^\circ$ is universal among symmetric graphs under $R$.
  \item The functor $R^+$ is universal among transitive graphs under $R$.
  \item The functor $R^*$ is universal among reflexive transitive graphs under $R$.
  \item The functor $R^\sym$ is universal among reflexive transitive symmetric graphs under $R$.
  \end{itemize}
\end{lemma}

\begin{example}
  In the case $\mathcal C = \Set$, a graph over a set $E$ is just a graph with
  set of objects $E$, and reflexivity and symmetry coincide with the usual
  terminology.

  If $\mathcal C$ is a Mal'cev variety, that is $\mathcal C$ is the category of
  models of an algebraic theory containing a ternary operation $t$ such that
  $t(x,x,y) = y$ and $t(x,y,y) = x$, then reflexivity implies symmetry and
  transitivity. Mal'cev varieties include $\mathcal C = \Gp$, the category of
  groups, whose where the ternary operation is given by $t(x,y,z) =
  xy^{-1}z$. Other examples also include the category of abelian groups, vector
  spaces, or the one of heaps.

  In the case of $\mathcal C = \Vect_k$, the category of $k$-vector spaces (or
  indeed any abelian category), then a truncation of the Dold-Kan correspondence
  shows that the category of reflexive graphs is equivalent to the category of
  chain complexes concentrated in degrees $0$ and $1$.
\end{example}


Moving from a graph $R$ to one of its closure does not change the equivalence
relation presented by $R$. This is made formal by the following lemma.
\begin{lemma}\label{lem:coeq_of_closures}
  If $R$ is a graph on an object $E$ of $\mathcal C$, let us denote by $E/R$ the
  coequaliser of the diagram $\source R, \target R : R \to E$.

  Then for any graph $R$ over $E$, there are canonical isomorphisms
  $E/R = E/R^* = E/R^\sym$.
\end{lemma}

\section{Internalizing termination}\label{sec:termination}

We fix a category $\mathcal C$ with all pullbacks and countable coproducts.

In this section, we define what it means for a graph to be terminating. This
property is expressed in term of a rewriting strategy. When $\mathcal C = \Set$,
such a strategy associates to any $x$ which is not a normal form a rewriting
step of source $x$. We show how the examples of terminating rewriting relations
of Section \ref{sec:abs_alg_rew} fit into the general definition, but also that
some non-terminating relations can be terminating as graphs.


\begin{note}
  If $I$ is partially ordered set, we silently see $I$ as a category, with an
  arrow $x \to y$ whenever $x \leq y$ in $I$.
\end{note}

\begin{definition}
  A \emph{directed set} is a non-empty partially ordered set $I$ such that for all
  $x,y \in I$, there exists $z$ such that $x,y \leq z$.
  
  Let $E$ be an object of $\mathcal C$. A \emph{directed structure} on $E$ is the data
  of a directed set $I$ together with a functor $E_\bullet : I \to \mathcal C$ and a
  natural transformation $\iota_\bullet : E_\bullet \to E$ exhibiting $E$ as the colimit of
  $E_\bullet$. If $I$ is a directed set, the data of $E_\bullet$ and
  $\iota_\bullet$ is called an \emph{$I$-filtration} of $E$, and we slightly abuse notations by
  saying that $E$ is an \emph{$I$-filtered object} if it comes equipped with an
  $i$-filtration $(E_\bullet, \iota_\bullet)$.

  A morphism of $I$-filtered object from $E$ to $F$ is the data of a natural
  transformation $\eta : E_\bullet \Rightarrow F_\bullet$. Note that this induces a map from
  $E$ to $F$.

  We denote by ${\Filt I}_{\mathcal C}$ the category of $I$-filtered objects in
  $\mathcal C$. 
\end{definition}

\begin{note}
  Identifying an object $E$ of $\mathcal C$ with a functor from the terminal
  category $\top \xrightarrow{E} \mathcal C$, the data of a filtered object in
  $\mathcal C$ fits into the following diagram:
  
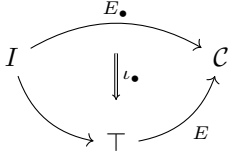
\[
    \begin{tikzcd}
      I
      \arrow[rr, bend left, "E_\bullet"]
      \arrow[dr, bend right]
      &
      \arrow[d, yshift=.3cm, -Implies, double, "\iota_\bullet"]
      & \mathcal C
      \\
      & \top \ar[ur, "E"', bend right] &
    \end{tikzcd}
  \]
\end{note}




\begin{definition}
  Let $I$ be a directed set. For any $i \in I$ we denote by $I_{<i}$ the
  subset of elements of $I$ smaller than $i$, and by $\minimum I$ the set of
  minimal elements of $I$.

  If $E$ is an $I$-filtered object of a category $\mathcal C$, then for any
  $i \in I$, we denote by $E_{<i}$ (resp. $\minimum E$) the colimit of the
  restriction of the functor $E_\bullet$ to $I_{<i}$ (resp. $\minimum I$). For
  $i \in \minimum I$, we define $E_{<i}$ as $E_i$. This defines a functor 
  $E_{<\bullet} : I \setminus \minimum I \to \mathcal C$. The universal property of the colimit
  induces a natural transformation $\iota_{<i} : E_{<i} \to E$.

  If $f : E \to F$ is a map in $\mathcal C$, it induces a natural transformation
  $f_\bullet : E_\bullet \Rightarrow F$.  For any $i \in I \setminus \minimum I$, restricting
  $f_\bullet$ to $I_{<i}$ and passing to the colimit induces a map
  $f_{<i} : E_{<i} \to F$, which we extend for $i \in \minimum I$ by setting
  $f_{<i} := f \circ \iota_i$. This defines a natural transformation
  $f_{<\bullet} : E_{<\bullet} \Rightarrow F$.
\end{definition}

\begin{note}
  By definition, $\minimum I$ is a discrete partially ordered set (i.e. for all
  $i,j \in \minimum I$, if $i \leq j$ then $i = j$). As a result $\minimum E$ is
  actually given by the coproduct $\minimum E = \coprod_{i \in \minimum I} E_i$.
\end{note}

\begin{definition}
  A terminating graph on an $I$-filtered object $E$ of a category $\mathcal C$
  is the data of a  graph $R \in \Graph_{E}$, together with a map
  $h : E \to R + E$ and a natural transformation $h^\tau_\bullet : E_\bullet \Rightarrow E_{<\bullet}$ such that
\begin{description}
\item[TG1] $I$ is a terminating order.
\item[TG2\label{itm:source_h}] For all $i \in I$, $\source {R+E} \circ h = id_E$.
\item[TG3\label{itm:h_normal_form}] For all $i \in \minimum I$, $h \circ \iota_i = u_{R+E} \circ \iota_i$.
\item[TG4\label{itm:target_h}] For all $i \in I$, $\target {R+E} \circ h \circ \iota_i = \iota_{<i} \circ h^\tau_i$.
\end{description}

The last two axioms can be represented by the following diagram:

    
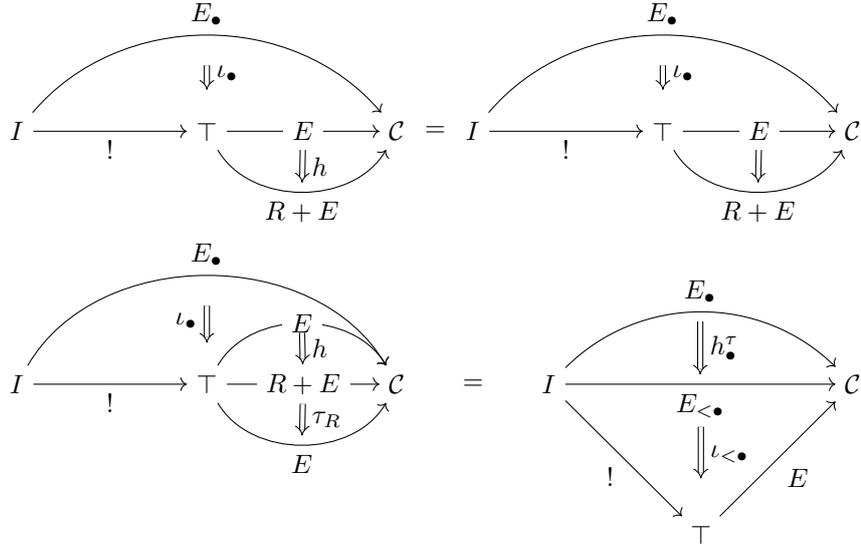
\begin{tikzpicture}
       \node (I) at (0,0) {$I$};
      \node (top) at (2.5,0) {$\top$};
      \node (C) at (5,0) {$\mathcal C$};
      \draw[->] (I) to[bend left=50] node [above] {$E_\bullet$} node (sh) {} (C);
      \draw[->] (I) -- node [below] {$!$} (top);
      \draw[->] (top) -- node [fill = white] (R) {$E$} (C);
      \doublearrow{-Implies, shorten >=.3cm, shorten <=.4cm}{(top|-sh) -- node [right] {$\iota_\bullet$} (top)};
      \draw[->] (top)  to[bend right=60] node[below] {$R+E$} node (E) {} (C);
      \doublearrow{-Implies}{(R) -- node [right] {$h$} (E)};
      \node (I') at (6,0) {$I$};
      \node (top') at (8.5,0) {$\top$};
      \node (C') at (11,0) {$\mathcal C$};
      \draw[->] (I') to [bend left=50] node [above] {$E_\bullet$} node (sh') {} (C');
      \draw[->] (I') -- node [below] {$!$} (top');
      \draw[->] (top')  to [bend right=60] node[below] {$R+E$} node (R') {} (C');
      \draw[->] (top') -- node [fill = white] (E) {$E$} (C');
      \doublearrow{-Implies, shorten >=.3cm, shorten <=.4cm}{(top'|-sh') -- node [right] {$\iota_{\bullet}$} (top')};
      \doublearrow{-Implies}{(E) -- node [right] {
        } (R')};
      \path (C) -- node {$=$} (I');
    \end{tikzpicture}
    
    \begin{tikzpicture}
      \node (I) at (0,0) {$I$};
      \node (top) at (2.5,0) {$\top$};
      \node (C) at (5,0) {$\mathcal C$};
      \draw[->] (I) to[bend left=60] node [above] {$E_\bullet$} node (sh) {} (C);
      \draw[->] (I) -- node [below] {$!$} (top);
      \draw[->] (top) -- node [fill = white] (R) {$R+E$} (C);
      \doublearrow{-Implies, shorten >=.3cm, shorten <=.4cm}{(top|-sh) -- node [left] {$\iota_\bullet$} (top)};
      \draw[->] (top)  to[bend left=60] node[fill=white] {$E$} node (E'') {} (C);
      \draw[->] (top)  to[bend right=60] node[below] {$E$} node (E) {} (C);
      \doublearrow{-Implies}{(E'') -- node [right] {$h$} (R)};
      \doublearrow{-Implies}{(R) -- node [right] {$\target R$} (E)};
      \node (I') at (7,0) {$I$};
      \node (C') at (11,0) {$\mathcal C$};
      \node (top') at (9,-2) {$\top$};
      \draw[->] (I') to[bend left=45] node [above] {$E_\bullet$} node (sht) {} (C');
      \draw[->] (I') to node [below] (E) {$E_{< \bullet}$} node[above] (tht){} (C');
      \draw[->] (I') to node [below left] {$!$} (top');
      \draw[->] (top') to node [below right] {$E$} (C');
      \doublearrow{-Implies}{(sht) -- node [right] {$h^\tau_\bullet$} (tht-|sht)};
      \doublearrow{-Implies, shorten >= .5cm}{(E.south -| top') -- node [above right] {$\iota_{<\bullet}$} (top')};
      \path (C) -- node {$=$} (I');
    \end{tikzpicture}
\end{definition}
\begin{note}
  Note also the last two axioms imply that for all $i \in \minimum I$,
  $h_i^\tau = id_{E_i}$.
\end{note}

\begin{example}\label{ex:non_terminating_terminating}
  While in terminating relations \emph{every} rewriting step is supposed to be
  decreasing, the definition of terminating graph defined above only requires
  that there exists a decreasing rewriting step of source $x$ whenever $x$ is
  not a normal form. For example, the following graph
  is terminating:
  
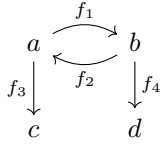
\[
    \begin{tikzcd}
      a \ar[r, "f_1", bend left]
      \ar[d,"f_3"']
      &
      b
      \ar[l, "f_2", bend left]
      \ar[d, "f_4"]
      \\
      c
      &
      d
    \end{tikzcd}
  \]
  Formally, we take $I = \{0,1\}$, $E_0 = \{c,d\}$ and $E_1 = E =
  \{a,b,c,d\}$. As a result, the set of normal forms is $\minimum E = \{c,d\}$.

  Defining a strategy $(h,h^\tau)$ on $R = \{f_1,f_2,f_3,f_4\}$ amounts to choosing
  for any $x \in E$ that is not a normal form a rewriting step
  $h(x) : x \to h^\tau(x)$. The fact that $h^\tau$ lands in $E_{<i}$ rather than
  $E$ witnesses the fact that $h^\tau(x)$ is in some sens ``smaller'' than $x$. On
  normal forms, axiom \refitem{itm:h_normal_form} forces $h(c) = 1_c$ and
  $h(d) = 1_d$, and by axiom \refitem{itm:target_h}, we necessarily have
  $h^\tau_0(c) = c$ and $h^\tau_0(d) = d$.

  For $i = 1$ by naturality of $h_\bullet^\tau$ we still have
  $h_1^\tau (c) = c$ and $h_1^\tau(d) = d$. Finally we
  set $h(a) = f_3$ and $h(b) = f_4$, which forces $h^\tau(a) = c$ and
  $h^{\tau}(b) = d$.

  We represent the situation by the diagram on the left hand side of the
  following picture. The horizontal bar denotes the filtration of $E$ induced by
  $I$, while the thick red arrows denote the arrows selected by the local
  strategy.
  
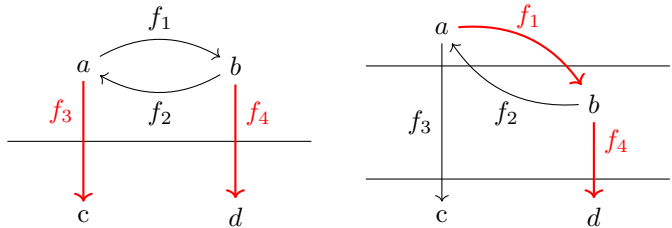
\[
    \begin{tikzpicture}
      \node (A) at (0,0){$a$};
      \node (B) at (2,0){$b$};
      \node (C) at (0,-2){c};
      \node (D) at (2,-2){$d$};
      \draw[-] (-1,-1) -- (3,-1);
      \draw[->] (A) to[bend left] node[above] {$f_1$} (B);
      \draw[->] (B) to[bend left] node[below] {$f_2$} (A);
      \draw[->, thick, red] (A) --node[near start, left]{$f_3$} (C);
      \draw[->, thick, red] (B) --node[near start, right] {$f_4$} (D);
    \end{tikzpicture}
    \qquad
     \begin{tikzpicture}
      \node (A) at (0,0.5){$a$};
      \node (B) at (2,-0.5){$b$};
      \node (C) at (0,-2){c};
      \node (D) at (2,-2){$d$};
      \draw[-] (-1,0) -- (3,0);
      \draw[-] (-1,-1.5) -- (3,-1.5);
      \draw[->, thick, red] (A) to[bend left] node[above] {$f_1$} (B);
      \draw[->] (B) to[bend left] node[below] {$f_2$} (A);
      \draw[->] (A) --node[left]{$f_3$} (C);
      \draw[->, thick, red] (B) --node[right, near start]{$f_4$} (D);
    \end{tikzpicture}
  \]
  The fact that $h^\tau_i$ lands in $E_{<i}$ is represented by the fact that the
  chosen arrows each go one step lower in the filtration.

  The diagram on the right hand side of the figure above pictures another
  filtration of $E$ (given by $I = \{0<1<2\}$, with $E_0 = \{c,d\}$,
  $E_1 = \{b,c,d\}$ and $E_2 = \{a,b,c,d\}$) and another local strategy, given by
   $h(a) =f_1$ and $h(b) = f_4$.
\end{example}

The following example explains how any terminating relation on a set $E$ induces
both a filtration of $E$, and a terminating graph on $E$.

\begin{example}\label{ex:set_terminates}
  Let $(E, \to)$ be a set equipped with a terminating relation. Then $E$ is
  naturally equipped with a structure of an $\mathbb N$-filtered set as follows:
  \[ E_0 = NF(\to) \qquad E_{i+1} = \{x \in E | \exists y \in E_i, x \xrightarrow{=} y \} \]
  Notice that for all $i \in \mathbb N$, $\mathbb N_{<i+1}$ has a terminal element
  (namely $i$), and so $E_{<i+1}$ is simply $E_i$.

  In addition, $\to$ induces a terminating graph $(R,h,h^\tau)$ on $E$, where:
  \begin{itemize}
  \item $R = \{(x,y) | x \to y\}$ with operations $\source R$ and
    $\target R$ given by the first and second projections respectively.
  \item For $x \in E_0$, i.e. $x$ is a normal form, we set $h^\tau_0 (x) = x$.
  \item For any $i \in \mathbb N$, take $x \in E_{i+¡}$. If $x \in E_i$, then we pose
    $h^\tau_{i+¡}(x) = h^\tau_i(x)$. Otherwise, by definition there exists
    $y \in E_i$ such that $x \xrightarrow{} y$, and we pose
    $h^\tau_{i+1}(x) = y$. 
  \item Finally, $h(x)$ is given by the pair $(x,h^\tau_i(x))$, which does not
    depend on $i$ by definition of $h^\tau_\bullet$.
  \end{itemize}
\end{example}

The next lemma will allow us to easily build filtered objects in categories
other than sets.

\begin{lemma}
  Let $F : \mathcal C \to \mathcal D$ be a cocontinuous functor.  Then for any
  directed set $I$, $F$ induces a functor
  \[ {\Filt I}_F : {\Filt I}_{\mathcal C} \to {\Filt I}_{\mathcal D}.\]

  In addition, for any $I$-filtered object $E$, $F(E)_i = F(E_i)$, $F(E)_{min} = F(E_{min})$
and $F(E)_{<i} = F(E_{<i})$.
\end{lemma}
\begin{proof}
  Let $E$ be an $I$-filtered object in $\mathcal C$.  Since
  $F$ preserves colimits, $(F(E),I,F \circ E_\bullet,F\circ \iota_\bullet)$ is a filtered object in
  $\mathcal D$. The image of maps by $F$ is as straightforward.
\end{proof}

\begin{example}
  Let $X$ be a set equipped with a terminating order $\leq$, and let $\to$ be an
  algebraic relation on $kX$ which is terminating with respect to $\leq$.

  To exhibit a filtration on $kX$, we first define an map
  $\height : X \to \mathbb N$, by induction on $x \in X$, using the order $\leq$. Without
  loss of generality, we can suppose that $NF(\to)$ is equal to the subspace of
  $kX$ spanned by the $x \in X$ which are minimal for $\leq$ (otherwise, replace
  $\leq$ by the relation $x \prec y$ defined by $x < y$ and $y\notin NF(\to)$).
  \begin{itemize}
  \item If $x$ is minimal for $\leq$, i.e. $x$ is a normal from, we define $\height(x) = 0$.
  \item Otherwise, we define:
    \[
      \height(x) := \min_{\substack{u \in kX \\ x \to u}} \left\{ \max_{y \in \supp(u)} \height (y) \right\} +1
    \]
  \end{itemize}

  We finally define a $\mathbb N$-filtration on $X$ by setting
  $X_i = \{x \in X | \height(x) \leq i\}$.  By the previous lemma, this filtration
  can be transported to $kX$. More precisely, a linear combination $u \in kX$ lies
  in $kX_i$ if and only if $\supp(u) \subseteq X_i$.

  Then $\to$ induces a reflexive graph $R$, which is the subspace of
  $kX \times kX$ generated by the pairs $(x,u)$ such that $x \xrightarrow{=} u$. To
  show that this graph is terminating, we define $h^\tau$ and $h$ as follows:
  \begin{itemize}
  \item Let $x \in X_0$, that is $x$ is a normal form. We then define
    $h^\tau(x) = x$, and extend it to $kX_0$ by linearity.
  \item For any $i \in I$, let $x \in X_{i+1}$. If $x \in X_i$, then we simply define
    $h^ \tau_{i+1}(x) = h_i^\tau(x)$. Otherwise, by definition there exists
    $u \in kX_i$ such that $x \to u$, and we define $h_{i+1}^\tau(x) = u$. We finally
    extend $h_{i+1}^\tau$ to $kX_{i+1}$ by linearity.
  \item For any $x \in X$, we finally define $h(x) = (x,h^\tau(x))$, and extend it to
    $kX$ by linearity.
  \end{itemize}
\end{example}

\begin{remark}
  Let $(R,h_\bullet,h_\bullet^\tau)$ be a terminating graph on an a filtered object
  $E$, and $S$ be an other graph over $E$. For any arrow $f : R \to S$,
  $(S,f \circ h_\bullet,h^\tau_\bullet)$ is a terminating graph over $E$. In particular, the
  various closures of $R$ defined in Section \ref{sec:graphs_as_relations} all
  inherit a canonical structure of terminating graph over $E$, which we still
  denote $h_\bullet$ and $h_\bullet^\tau$.
\end{remark}

\section{Confluence and strategies}\label{sec:confluence}
In this section, we define the notion of confluence and local confluence of a
graph.  Our criterion for confluence is that a for a confluent and terminating
graph, the quotient by the graph should be isomorphic to the object of normal
forms. We show this property, and prove that local confluence together with
termination imply confluence: this is Newman's Lemma.

\begin{definition}
  Let $\mathcal C$ be any category. A split coequalizer in $\mathcal C$ is a
  diagram
   \[
    \begin{tikzcd}[column sep = large]
      A \ar[r, "f", yshift=.1cm] \ar[r, ,yshift=-.1cm, "g"']
      &
      B \ar[r, "e"]
      \ar[l,bend right = 50, "t"']
      &
      C
      \ar[l, bend right = 50, "s"']
    \end{tikzcd}
  \]
  such that:
  \[
    e \circ f = e \circ g \qquad e \circ s = id_C \qquad s \circ e = g \circ t \qquad f \circ t = id_B
  \]
  For any such diagram, $e$ is necessary a coequalizer of $f$ and $g$, which we
  denote by $B/A = C$. In addition, such a coequaliser is absolute, meaning that
  it is preserved by any functor.
\end{definition}

\begin{definition}
  Let $R$ be a graph on a filtered object $E$. A \emph{(global) strategy} for $R$ is the
  data of a pair of morphisms $H : E \to R$ and $H^\tau : E \to \minimum E$ such that
  the following equations hold:
  \[
    \source R \circ H = id_E \qquad
    \target R \circ H = \minimum \iota \circ H^{\tau} \qquad
    H^{\tau} \circ \minimum \iota = Id_{\minimum E}
  \]

  A strategy $(H,H^\tau)$ is said to be \emph{confluent} if
  $H^\tau \circ \source R = H^\tau \circ \target R$
\end{definition}

The following proposition immediately follows from the definition of confluent
strategy.
\begin{proposition}\label{prop:conf_implies_split_coeq}
  Let $R$ be a graph over a filtered object $E$. If there exists a confluent
  strategy $(H,H^\tau)$ for $R$ then the diagram

   \[
    \begin{tikzcd}[column sep = large]
      R \ar[r, "\source R", yshift=.1cm] \ar[r, ,yshift=-.1cm, "\target R"']
      &
      E \ar[r, "H^\tau"]
      \ar[l,bend right = 50, "H"']
      \ar[r, bend left = 50, hookleftarrow, "\minimum \iota"]
      &
      \minimum E
    \end{tikzcd}
  \]
 forms a split coequalizer. In particular, $E / R = \minimum E$.
\end{proposition}


\begin{theorem}
  Let ($R, h_\bullet, h^\tau_\bullet)$ be a terminating graph on an $I$-filtered object
  $E$.  

  Then there is a strategy $(H,H^\tau)$ for $R^\sym$, which satisfy the additional
  equations:
  \[
    H = \mu_{R} \circ \langle h , H \circ \target{R+E} \circ h \rangle \qquad H^\tau = H^\tau \circ \target{R+E} \circ h
  \]
  Where $\mu_R$ denotes the canonical map $R R^\sym \to R^\sym$.
\end{theorem}
\begin{proof}
  Since $E$ is the colimit of the functor $E_\bullet$, it is enough to define natural
  transformations $H_\bullet :E_\bullet \Rightarrow R^\sym$ and
  $H^\tau_\bullet : E_\bullet \Rightarrow \minimum E$ satisfying the following equations: 
  \[
    \source {R^\sym} \circ H_\bullet = \iota_\bullet
    \qquad
    \target {R^\sym} \circ H_\bullet = \minimum \iota \circ H_\bullet^{\tau}
    \qquad
    H^{\tau}_i = \iota_i \text{, where $i \in \minimum I$.} 
  \]
  \[
    H_\bullet = \mu_{R} \circ \langle h \circ \iota_\bullet, H \circ \target{R+E} \circ h \circ \iota_\bullet\rangle \qquad H^\tau_\bullet = H^\tau \circ \target{R+E} \circ h \circ \iota_\bullet
  \]
  Note that using the equations relating $h$ to $h^\tau$, the last two relations can be rewritten as: 
  \[
    H_\bullet = \mu_{R} \circ \langle h \circ \iota_\bullet, H_{<\bullet} \circ h_\bullet^\tau \rangle \qquad H^\tau_\bullet = H^\tau_{<\bullet} \circ h_\bullet^\tau
  \]

  We proceed by induction on $i \in I$ to build such natural transformations. If
  $i$ is minimal, then the third equation entirely determines $H^\tau_i$, and we
  take $H_i := u_{R^\sym} \circ \iota_i$.  Then the first equation  holds because
  $\source {R^\sym} \circ H_i = \source {R^\sym} \circ u_{R^\sym} \circ \iota_i = \iota_i$, while for the second we have
  $\target {R^\sym} \circ H_i = \iota_i = \minimum \iota \circ \iota_i$, where the second equality holds
  because $i$ is minimal.

  Suppose now that $H_\bullet$ and $H^\tau_\bullet$ are defined for all $j < i$, naturally in
  $j$. Then they induce functors $H_{<i} : E_{<i} \to R^\sym$ and
  $H^\tau_{<i} : E_{<i} \to \minimum E$. By induction hypothesis those satisfy the
  following equations:
  \[
  \source {R^\sym} \circ H_{<i} = \iota_{<i}
    \qquad
    \target {R^\sym} \circ H_{<i} = \minimum \iota \circ H_{<i}^{\tau}
    \qquad
    H_{<i}^\tau \circ \iota_j = \iota_j \text{, where $j \in \minimum I$.} 
  \]

  Before defining $H_i$ let us consider the two maps $h_i$ and
  $H_{<i} \circ h_i^\tau$, from $E_i$ to respectively $R$ and $R^\sym$.  Notice that we have
  $\target R \circ h_i = \iota_{<i} \circ h^\tau_i = \source {R^\sym} \circ H_{<i} \circ h^\tau_i$, and so these
  two maps induce a map $\langle h_i , H_{<i} \circ h_i^\tau \rangle : E_i \to RR^\sym$. We finally define:
  \[
    H_i := m_R \circ \langle h_i , H_{<i} \circ h_i^\tau \rangle \qquad H^\tau_i := H^\tau_{<i} \circ h^\tau_i.
  \]

  The fact that $H_i$ and $H_i^\tau$ satisfy the required equations in all cases is then a straightforward verification.

\end{proof}




\section{From local to global confluence}\label{sec:loc_to_glob_confluence}
The goal of this section is to give a local confluence criteria in order to
prove confluence. In this section, we suppose that $\mathcal C$ is a locally
finitely presentable category.

\begin{definition}
  Let $(R,h,h^\tau)$ be a terminating graph over a filtered object $E$. A
  \emph{local-confluence structure} (or \emph{lc-structure} for short) on
  $(R,h,h^\tau)$ is the structure of a $J$-filtered object on $R$, where $J$ is the
  category associated to a total terminating order such that
  $\minimum R = \emptyset$, and a natural transformation
  $c_\bullet : R_\bullet \Rightarrow R_{<\bullet}^\sym$ satisfying the equations:
  \[
    \source {R^\sym} \circ \iota^\sym_{<\bullet} \circ c_\bullet = \target{R+E} \circ h \circ \source {R_\bullet} 
    \qquad
    \target {R^\sym} \circ \iota^\sym_{<\bullet} \circ c_\bullet = \target {R_\bullet} 
\]
\end{definition}

The reason for the restriction to the case where $\mathcal C$ is locally
finitely presentable, and $J$ is total is the following lemma, that we
implicitly use in the proof of Theorem \ref{thm:Newman_general}.

\begin{lemma}
  Let $\mathcal C$ be a locally finitely presented category, and let $R$ be a
  graph on a filtered object $E$. Suppose that $R$ is $J$-filtered. Then
  \begin{itemize}
  \item The functor $(R_\bullet)^* : J \to \mathcal C$ is a $J$-filtration of $R^*$.
  \item If $J$ is total, then for any $j \in J$, $(R_{<j})^\sym = (R^{*})_{<j}$.
  \end{itemize}
In addition, the same properties hold for $R + R^\circ$, and thus for $R^\sym$.
\end{lemma}
\begin{proof}
  Since colimits commute with colimits, we just have to prove that
  $\colim_j R_j^{\times_E n} = R^{\times_E n}$.  For $n = 0,1$ this is clear. Let us treat
  the case $ n = 2$, the general case being similar. Note first that since $J$
  is directed the inclusion $J \to J \times J$ is final, and so
  $ \colim_j R_j \times_E R_j = \colim_{i,j} R_i \times_E R_j $. Finally in a locally
  finitely presentable category pullbacks preserve directed colimits
  \cite[Proposition 1.59]{Ada_Rosick}, and so we get:
  \[
    \colim_j R_j \times_E R_j =  \colim_i (\colim_j (R_i \times_E R_j)) = \colim_i (R_i \times_E R) = R \times_E R
  \]

  To compute $(R_{<j})^*$ we use the same technique, using the face that since
  $J$ is total then $I_{<j}$ is still directed.

  The case of $R + R^\circ$ is clear since colimits commute with colimits, and the
  case of $R^\sym$ follows by combining the two previous cases.
\end{proof}

\begin{theorem}[Newman's Lemma for graphs]\label{thm:Newman_general}
  Let $(R,h,h^\tau)$ be a terminating graph over a filtered object $E$, equipped
  with an lc-structure. Then the strategy induced by $(R,h,h^\tau)$ on $R^\sym$ is
  confluent.

  In particular, $E/R = E/R^\sym  = \minimum E$.
\end{theorem}
\begin{proof}
 Let us prove by induction on $j \in J$ that the equality
  $H^\tau \circ \source {R_j^\sym}  = H^\tau \circ \target {R_j^\sym}$ holds. 

  Suppose first that $j$ minimal. Since
  $R_{min} = \coprod_{j \in \minimum J} R_j = \emptyset$, then
  $R_j = \emptyset$. The required equation thus holds by unicity of the maps from the
  initial object.
  
  Otherwise, by induction we have
  $H^\tau \circ \source {R_{<j}^\sym} = H^\tau \circ \target {R_{<j}^\sym}$. Then we get
  $H^\tau \circ \source {R_j} = H^\tau \circ \target {R_j}$. Indeed we have:
  \begin{align*}
    H^\tau \circ  \source {R_j} 
    & = H^\tau \circ \target{R+E} \circ h \circ  \source {R_j}
    \\ & = H^\tau \circ  \source {R^\sym} \circ \iota^\sym_{<j} \circ c_j
    \\ & = H^\tau \circ \target {R_j^\sym} \circ \iota_{<j} \circ c_j
    \\ & = H^\tau \circ \target {R_j}
  \end{align*}
  By exchanging the roles of $\source{R_j}$ and $\target{R_j}$, the same
  equality holds for $R_j^\circ$. Let us denote by $T$ the sum $R_j + R_j^\tau$ and by
  $T^n$ the limit $T^{\times_E n}$, and let us show that for all
  $n \in \mathbb N$,
  $H^\tau \circ \source{T^{n}} = H^\tau \circ \target{T^{n}}$. For $n = 1$ this is the
  previous discussion. For $n = 0$ $\source{T^0}$ and $\target{T^0}$ coincide
  and the equality holds. Finally for $n \geq 2$ we have, using repeatedly  $n = 1$:
  \[
    H^\tau \circ \source{T^n} = H^\tau \circ \source{T} \circ \pi_1 = H^\tau \circ \target{T} \circ \pi_1 = H^\tau \circ
    \source{T} \circ \pi_2 = \ldots = H^\tau \circ \target{T} \circ \pi_n = H^\tau \circ \target{T^n}.
  \]
  
  Putting all those equalities together we finally get
  $H^\tau \circ \source{R^\sym_j} = H^\tau \circ \target{R^\sym_j}$, and by colimit
  $H^\tau \circ \source{R^\sym} = H^\tau \circ \target{R^\sym}$. Hence $H$ is indeed a
  confluent strategy on $R^\sym$. The equality $E/R^\sym = \minimum E$ follows
  by Proposition \ref{prop:conf_implies_split_coeq}, and $E/R = E/R^\sym$ by Lemma \ref{lem:coeq_of_closures}.
\end{proof}

\begin{example}
  Let $(E,\to)$ be a set equipped with a terminating relation. Following Example
  \ref{ex:set_terminates}, we can equip $E$ with an $\mathbb N$-filtered structure, and $\to$ induces a
  terminating graph $R = \{(x,y)|x \to y\}$ on $E$. Then there is a canonical
  filtration of $R$ given by the following pullback:
  \[
    \begin{tikzcd}
      R_j
      \ar[r]
      \ar[d]
      &
      E_j
      \ar[d,"\iota_j"]
      \\
      R
      \ar[r,"\source R"]
      & E.
    \end{tikzcd}
  \]
  Since $\Set$ is locally finitely presented, pullbacks preserve directed
  colimits and thus $\colim_j R_j = R$. More precisely,
  $(x,y) \in R$ lies in $R_j$ if and only if $x$ lies in $E_j$. In particular
  $R_0 = \emptyset$ since $E_0$ is the set of normal forms of $\to$.

  Then the existence of an lc-structure $c_\bullet$ on $R$ is equivalent to the
  existence for any $r = (x,y) \in R$ of some equivalence path
  $c(r) : h^\tau(x) \xleftrightarrow{*} y$ of shape:
  \[
    \begin{tikzcd}
      x \ar[r, "r"] \ar[d,"h(x)"'] & y \\
      h^\tau(x) \ar[ru, "c(r)"'] & 
    \end{tikzcd}
  \]
  such that any rewriting step occurring in $c(x)$ has source smaller than
  $x$. This is true whenever $\to$ is locally confluent, and Theorem
  \ref{thm:Newman_general} recovers that $R$ is confluent.

  The same phenomenon applies to the case $\mathcal C = \Ab$, which is also
  locally finitely presented.
\end{example}

\begin{note}
  The idea of ordering the relations is not new in abstract rewriting, and is
  for example the main idea behind van Oostrom's notion of decreasing diagrams
  \cite{VO_decreasing}. In fact when specialised to the case
  $\mathcal C = \Set$, Theorem \ref{thm:Newman_general} (which is more general
  than Newman's Lemma since, as noted in Example
  \ref{ex:non_terminating_terminating}, termination in the sense of Definition
  XX is more geenral than the ususal termination of relations) is a direct
  consequence of \cite[Theorem 3.7]{VO_decreasing}
\end{note}


 
\bibliography{biblio}

\end{document}